\documentclass[10pt, oneside]{amsart}
\usepackage{amsfonts}
\usepackage{verbatim}
\usepackage{hyperref}
\usepackage{tikz}
\usepackage{amsmath,amssymb}

\graphicspath{ {./images/} }
 
\begin{document}
 
\title{Dimensions of Mycielskians of cycles}
\author{Brian Chung}
\address{Palo Alto High School}
\email{brianchungus0721@gmail.com}

\author{Mike Krebs}
\address{Department of Mathematics, California State University --- Los
Angeles}
\email{mkrebs@calstatela.edu}

\begin{abstract}
The Mycielskian is a standard construction studied in many an introductory graph theory course.  It is natural to consider Mycielskians of cycles, some of the simplest of all graphs.  This paper deals with the so-called ``dimension'' of such graphs.  The dimension of a graph $G$ is the smallest positive integer $n$ such that there exists a one-to-one correspondence between the vertices of $G$ and some collection of points in $n$-dimensional Euclidean space such that if two vertices in $G$ are adjacent, then the distance between the corresponding points is $1$.  In previous works, it had been proven that the dimension of the Mycielskian of a $k$-cycle is 3 when $k$ is $3$, $4$, or $5$, and 2 when $k=10$.  In this paper, we answer the question completely.  Namely, we show that the dimension is $3$ when $k\neq 10$, and $2$ when $k=10$.
\end{abstract}
 
\maketitle
 
 
\newtheorem{Mthm}{Main Theorem\!\!\,}
\newtheorem{Thm}{Theorem}[section]
\newtheorem{Prop}[Thm]{Proposition}
\newtheorem{Lem}[Thm]{Lemma}
\newtheorem{Cor}[Thm]{Corollary}
\newtheorem{Guess}[Thm]{Conjecture}
 
\theoremstyle{definition}
\newtheorem{Def}[Thm]{Definition}
 
\theoremstyle{remark}
\newtheorem{Ex}[Thm]{Example}
\newtheorem{Rmk}[Thm]{Remark}
\newtheorem{Not}[Thm]{Notation}
 
\renewcommand{\theThm} {\thesection.\arabic{Thm}}
\renewcommand{\theProp}{\thesection.\arabic{Prop}}
\renewcommand{\theLem}{\thesection.\arabic{Lem}}
\renewcommand{\theCor}{\thesection.\arabic{Cor}}
\renewcommand{\theDef}{\thesection.\arabic{Def}}
\renewcommand{\theGuess}{\thesection.\arabic{Guess}}
\renewcommand{\theEx}{\thesection.\arabic{Ex}}
\renewcommand{\theRmk}{}
\renewcommand{\theMthm}{}
\renewcommand{\theNot}{}
\renewcommand{\thefootnote}{\fnsymbol{footnote}}


\section{Introduction}\label{sec:intro}

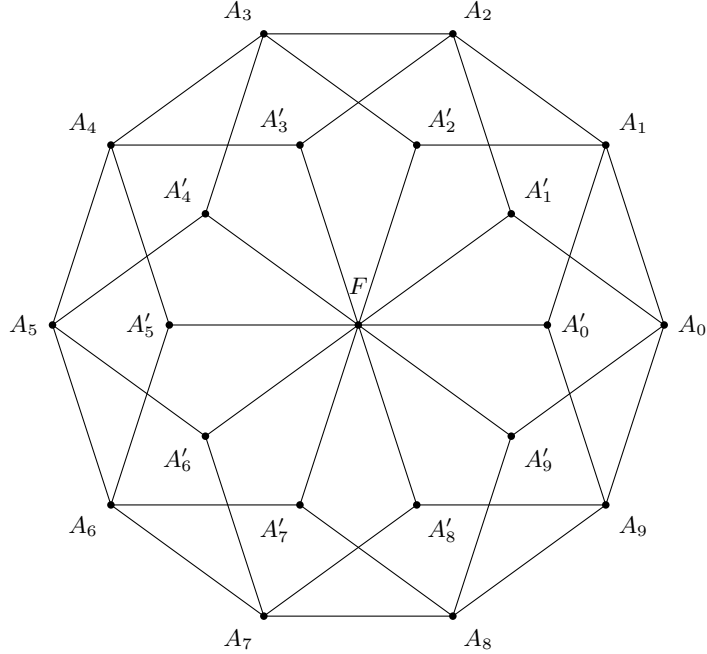
\begin{figure}\label{fig:M-of-C-10}
\centering
\begin{tikzpicture}[scale=2.5, every node/.style={font=\small}]
    \def\rp{1}
    \def\ru{1.61803398875}

    \coordinate (F) at (0,0);

    \foreach \j in {0,...,9} {
        \pgfmathsetmacro{\ang}{36*\j}
        \coordinate (P\j) at (\ang:\rp);
        \node[
            inner sep=1pt,
            fill=black,
            circle,
            label={\ang:$A_{\j}'$}  
        ] at (P\j) {};
    }

    \foreach \j in {0,...,9} {
        \pgfmathsetmacro{\ang}{36*\j}
        \coordinate (U\j) at (\ang:\ru);
        \node[
            inner sep=1pt,
            fill=black,
            circle,
            label={\ang:$A_{\j}$}  
        ] at (U\j) {};
    }

    \node[inner sep=1pt, fill=black, circle] at (F) {};
    \node[above=8pt] at (F) {$F$};

    \foreach \j in {0,...,9} {
        \draw (F) -- (P\j);
    }

    \foreach \j in {0,...,9} {
        \pgfmathtruncatemacro{\jp}{mod(\j+1,10)}
        \draw (U\j) -- (U\jp);
    }

    \foreach \j in {0,...,9} {
        \pgfmathtruncatemacro{\jm}{mod(\j+9,10)}
        \pgfmathtruncatemacro{\jp}{mod(\j+1,10)}
        \draw (U\j) -- (P\jm);
        \draw (U\j) -- (P\jp);
    }

\end{tikzpicture}

\caption{A unit-distance embedding of the Mycielskian of a $10$-cycle}
\end{figure}

Consider the image shown in Figure \ref{fig:M-of-C-10}.  Its ``outer ring'' consists of the ten points $A_0,\dots,A_9$.  Arranging them equiangularly around a circle of radius $(1+\sqrt{5})/2$ centered at $F$, we have that the distance from $A_j$ to $A_{j+1}$ is $1$ for all $j$.  (Here the subscript is taken modulo $10$.)  Drawing an edge between two points when the distance between them is $1$, the points $A_0,\dots,A_9$ then form the vertex set of a cycle graph with $10$ vertices, which we denote $C_{10}$.

In the ``inner ring'' in this figure we have then ten points $A_0',\dots,A_9'$.  These we place as shown on the circle of radius $1$ centered at $F$.  Observe that each point $A_j'$ is at distance $1$ to $A_{j-1}$ and $A_{j+1}$, which were precisely the neighbors of $A_j$ in the $10$-cycle formed by the ``outer ring.''

Finally, by our choice of radius, we have that $F$ is at unit distance from each point $A_j'$.  This now accounts for all of the edges shown in Figure \ref{fig:M-of-C-10}.

\vspace{.05in}

Graph theorists will recognize this figure as a special case of the Mycielskian construction.  We begin with a finite graph $G$ with vertex set $A_1,\dots,A_k$.  We then make a ``copy'' $A_1',\dots,A_k'$ of this vertex set, and we declare the neighbors of $A_j'$ to be precisely those $A_m$ adjacent to $A_j$.  Finally, we add in one additional vertex $F$, and an edge between $F$ and every vertex $A_j'$.  The resulting graph is the Mycielskian of $G$, denoted $M(G)$.  (Some authors prefer $\mu(G)$ instead.)

\vspace{.05in}

Figure \ref{fig:M-of-C-10} shows that there is a unit-distance embedding of $M(C_{10})$ in $\mathbb{R}^2$.  For which $k$ is such an embedding possible for $M(C_k)$?  Note that there is no requirement that the vertices $A_0,\dots,A_{k-1}$ (or the vertices $A_0',\dots,A_{k-1}'$) be spaced evenly around a circle.  All vertices may be placed anywhere whatsoever in the plane, so long as edges correspond to unit distances.

More generally, for a given $k$, what is the dimension of $M(C_k)$, that is, the smallest $m$ such that there exists a unit-distance embedding of $M(C_k)$ in $\mathbb{R}^m$?  We denote this by $\text{dim}(M(C_k))$.  (We note that there are some nuances with this definition --- see Section \ref{sec:dim} for the precise details.)

\vspace{.05in}

Some partial results were previously known. The book \cite{soifer2024new}, for example, gives the proof that $\text{dim}(M(C_{10}))=2$; this proof consists of Figure \ref{fig:M-of-C-10} and our discussion of it.    In \cite{hudak2016note}, it is shown that the $\text{dim}(M(C_k))=3$ when $k$ is $3$, $4$, or $5$; and that $\text{dim}(M(C_k))\leq 3$ when $k\geq 6$.  The result for $k=5$ is also contained in \cite{hong2014euclidean}.  A claim is made in \cite{soifer2024new} that $\text{dim}(M(C_k))\geq 3$ for all odd values of $k$; however, no proof is provided.

Our main theorem completely determines $\text{dim}(M(C_n))$ for all $n$.

\begin{Thm}\label{thm:main}
    Let $n\geq 3$ be an integer.  Then:\[
\dim(M(C_n))=\begin{cases}
    2 & \text{ if }n=10\\
    3 & \text{ if }n\neq 10
\end{cases}
\]
\end{Thm}

The remainder of this paper is primarily devoted to making precise all preliminary details in the run-up to this theorem, and then a proof of the theorem.  The method of proof generalizes, and we conclude the paper with a theorem that provides a necessary condition for an arbitrary graph to have dimension $\leq 2$.

\subsection{Elementary notions from graph theory}

In this subsection we provide some graph theory basics; readers familiar with them may want to skip ahead to the next subsection.

A (simple) \emph{graph} consists of two sets $V$ and $E$, where every element of $E$ is a two-element set of $V$.  Elements of $V$ (the \emph{vertex set}) are called \emph{vertices}, and elements of $E$ (the \emph{edge set}) are called \emph{edges}.  We denote the vertex set of a graph $X$ by $V(X)$, and its edge set by $E(X)$.  We visualize vertices as dots, and we visualize an edge $\{a,b\}$ as a line segment or curve with $a$ and $b$ as endpoints.  We say that $a$ and $b$ are \emph{adjacent}, denoted $a\sim b$, if $\{a,b\}\in E$.

The so-called \emph{cycle graphs} $C_n$ comprise one standard family of graphs.  The vertex set of $C_n$ is $\{0,1,\dots,n-1\}$, where $a$ and $b$ are adjacent if and only if $a-b\equiv\pm 1\;(\text{mod }n)$.

Two graphs $G$ and $H$ are \emph{isomorphic} if there is a bijective function $f\colon G\to H$ from the vertex set of one to the vertex set of another such that $a$ and $b$ are adjacent in $G$ if and only if $f(a)$ and $f(b)$ are adjacent in $H$.  We say that a graph $H$ is a \emph{subgraph} of a graph $G$ if $V(H)\subset V(G)$ and $E(H)\subset E(G)$.

\subsection{(Unit-distance) dimension of a graph}\label{sec:dim}

We define $U_n$ to be the graph with $\mathbb{R}^n$ as its vertex set, where two vertices are adjacent if and only if the Euclidean distance between them is $1$.  We say that a graph is a \emph{unit-distance graph} if it is isomorphic to a subgraph of $U_2$.

In \cite{erdos1965dimension}, Erdős, Harary, and Tutte introduce the notion of the ``dimension'' of a graph.  That is, they define the \emph{(unit-distance) dimension} of a graph $G$, denoted $\dim(G)$, to be the minimum positive integer $n$ such that $G$ is a subgraph of $U_n$.  (Although \cite{erdos1965dimension} simply calls this the ``dimension'' as do many other publications that followed suit, we include the parenthetical ``unit-distance'' here to disambiguate this definition from other competing notions of dimension for graphs.)

\begin{Ex}
    For an integer $n\geq 3$, we define the \emph{cycle graph} $C_n$ to be the graph with vertex set $\{1,\dots,n\}$, where two vertices are adjacent if and only if their difference is $\pm 1$ modulo $n$.  It is straightforward to show that no subgraph of $U_1$ is isomorphic to a cycle graph.  However, one obtains a subgraph of $U_2$ isomorphic to $C_n$ by considering a regular $n$-gon of side length $1$.  In this way $C_3$ ``looks like'' an equilateral triangle, $C_4$ a square, and so on.  It follows that $\dim(C_n)=2$.
\end{Ex}

\begin{Ex}
For an integer $n\geq 2$, we define the \emph{complete graph} $K_n$ to be the graph with vertex set $\{1,\dots,n\}$, where every vertex is adjacent to every other vertex.  It is relatively straightforward to prove that $\dim(K_n)=n-1$.  The embedding in $U_{n-1}$ is achieved by taking the vertices of a regular $n$-simplex with side length $1$.  So $K_2$ ``looks like'' a line segment, $K_3$ an equilateral triangle, $K_4$ a regular tetrahedron, and so on.
\end{Ex}

We remark that there is a definition related to that of (unit-distance) dimension, namely, \emph{Euclidean dimension}.  Here, one insists that the subgraph of $U_n$ be an induced subgraph.  But Euclidean dimension is not the subject of the present paper, so we shall say nothing more of it.

\subsection{Unit-distance dimension of Mycielskians of cycles} The ``Mycielskian'' is a standard construction in graph theory.  It was originally used in \cite{Mycielski1955} to construct a family of triangle-free graphs with arbitrarily large chromatic number.  Since then it has been much studied, to put it mildly --- as of the time of writing, Semantic Scholar lists 584 citations for \cite{Mycielski1955}.

\begin{Def}
    Let $G$ be a graph.  We define $M(G)$, the \emph{Mycielskian} of $G$, to be the following graph.  The vertex set of $M(G)$ is $\{F\}\cup (V(G)\times\{0,1\})$, where $F\notin V(G)\times\{0,1\}.$  (We think of the vertex set of $M(G)$ as consisting of two copies of $V(G)$, plus one additional vertex $F$.)  For adjacent vertices $a,b\in V(G)$, we have that $(a,0)\sim(b,1)$ and $(a,1)\sim(b,1)$.  Moreover, we have that $F\sim(a,0)$ for all $a\in V(G)$.  There are no other edges in $M(G)$.
\end{Def}

\begin{figure}[ht]
\centering
\begin{tikzpicture}[scale=1.4, every node/.style={font=\small}]
    \def\yb{0}      
    \def\ym{1.6}    
    \def\yt{3.2}    

    \def\s{1.4}

    \foreach \j in {0,...,4} {
        \coordinate (U\j) at ({\s*\j},\yb);
        \node[inner sep=1pt, fill=black, circle] at (U\j) {};
        \node[below=4pt] at (U\j) {$A_{\j}$};
    }

    \foreach \j in {0,...,4} {
        \coordinate (P\j) at ({\s*\j},\ym);
        \node[inner sep=1pt, fill=black, circle] at (P\j) {};
    }

    \node[above=4pt] at (P0) {$A_0'$};
    \node[above=4pt] at (P1) {$A_1'$};
    \node[right=6pt] at (P2) {$A_2'$}; 
    \node[above=4pt] at (P3) {$A_3'$};
    \node[above=4pt] at (P4) {$A_4'$};

    \coordinate (F) at ({2*\s},\yt);
    \node[inner sep=1pt, fill=black, circle] at (F) {};
    \node[above=4pt] at (F) {$F$};

    \foreach \j in {0,...,4} {
        \draw (F) -- (P\j);
    }

    \foreach \j in {0,...,3} {
        \pgfmathtruncatemacro{\jp}{\j+1}
        \draw (U\j) -- (U\jp);
    }

    \draw[bend left=40] (U4) to (U0);

    \foreach \j in {0,...,4} {
        \pgfmathtruncatemacro{\jm}{mod(\j+4,5)}
        \pgfmathtruncatemacro{\jp}{mod(\j+1,5)}
        \draw (U\j) -- (P\jm);
        \draw (U\j) -- (P\jp);
    }
\end{tikzpicture}
\caption{The Mycielski–Grötzsch graph}
\label{fig:Grotzsch}\end{figure}
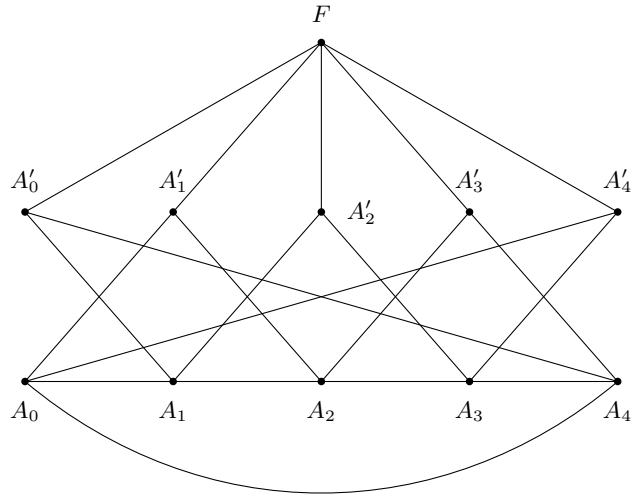

\begin{Ex}\label{Ex-C-5} The graph $M(C_5)$, shown in Figure \ref{fig:Grotzsch}, is called the \emph{Mycielski–Gr\"otzsch graph}.  In the labels in that figure, we write $A_j$ for $(j,1)$ and $A'_j$ for $(j,0)$.\end{Ex}

\begin{Ex}\label{Ex-C-10} Figure \ref{fig:M-of-C-10} shows $M(C_{10})$.  As in the previous example, we write $A_j$ for $(j,1)$ and $A'_j$ for $(j,0)$.  Let $\phi=(1+\sqrt{5})/2$.  For convenience, we identify $\mathbb{R}^2$ with $\mathbb{C}$.  We obtain an isomorphism to a subgraph of $U_2$ by mapping $F\mapsto0$, and 
$A_j'\mapsto e^{\pi ij/5}$, and $A_j\mapsto \phi e^{\pi ij/5}$.  The locations of the vertices in Figure \ref{fig:M-of-C-10} were chosen accordingly.\end{Ex}
 
\section{Proof of Theorem \ref{thm:main}}

In this section, we prove Theorem \ref{thm:main}.  We begin with an upper bound.

\begin{Lem}[\cite{hudak2016note}]\label{lem:upper-bound}Let $n\geq 3$ be an integer.  Then $\dim(M(C_n))\leq 3$.
\end{Lem}

We give here the core idea of the proof of Lemma \ref{lem:upper-bound}.  Imagine $M(C_n)$ arranged essentially as in Figure \ref{fig:M-of-C-10}, though with edges not necessarily of unit length.  The vertices are arranged in three layers: an outer ring, comprised of the vertices $A_j$; an inner ring, comprised of the vertices $A_j'$; and the center vertex $F$.  By placing these layers at different ``levels'' in $\mathbb{R}^3$, i.e., with different $z$-coordinates, one obtains a unit-distance embedding.  Figure \ref{fig:M-of-C-7} depicts this construction for $M(C_7)$.  Vertices in the outer ring are shown as black dots; those in the inner ring as grey dots; and $F$ as a hollow dot.

\begin{figure}[ht]
\centering
\begin{tikzpicture}[
    scale=4.3,
    x={(1cm,0cm)},
    y={(0.45cm,0.25cm)},
    z={(0cm,1cm)},
    outervertex/.style={circle, fill=black, draw=black, inner sep=1.4pt},
    innervertex/.style={circle, fill=gray!45, draw=black, inner sep=1.4pt},
    fvertex/.style={circle, fill=white, draw=black, inner sep=1.6pt}
]

\def\n{7}
\pgfmathsetmacro{\ru}{1/(2*sin(180/\n))}
\def\rp{1}
\pgfmathsetmacro{\h}{sqrt(1 - (\ru*\ru + \rp*\rp
    - 2*\ru*\rp*cos(360/\n)))}

\coordinate (F) at (0,0,0);

\foreach \j in {0,...,6} {
    \pgfmathsetmacro{\ang}{360*\j/\n}
    \coordinate (P\j) at ({\rp*cos(\ang)},{\rp*sin(\ang)},0);
    \coordinate (U\j) at ({\ru*cos(\ang)},{\ru*sin(\ang)},\h);
}

\draw[gray!35, dashed] plot[domain=0:360, samples=80]
    ({\rp*cos(\x)},{\rp*sin(\x)},0);

\draw[gray!35, dashed] plot[domain=0:360, samples=80]
    ({\ru*cos(\x)},{\ru*sin(\x)},\h);

\foreach \j in {0,...,6} {
    \draw (F) -- (P\j);
}

\foreach \j in {0,...,6} {
    \pgfmathtruncatemacro{\jp}{mod(\j+1,\n)}
    \draw[thick] (U\j) -- (U\jp);
}

\foreach \j in {0,...,6} {
    \pgfmathtruncatemacro{\jm}{mod(\j+\n-1,\n)}
    \pgfmathtruncatemacro{\jp}{mod(\j+1,\n)}
    \draw (U\j) -- (P\jm);
    \draw (U\j) -- (P\jp);
}

\node[fvertex] at (F) {};
\foreach \j in {0,...,6} {
    \node[innervertex] at (P\j) {};
    \node[outervertex] at (U\j) {};
}

\node[below left=2pt] at (F) {\small $F$};

\node at ({1.55*\ru},{0.15},{\h+0.12})
    {\small outer ring};

\node at ({1.45*\rp},{-0.1},{-0.08})
    {\small inner ring};

\end{tikzpicture}

\caption{A unit-distance embedding of $M(C_7)$ in $\mathbb R^3$}
\label{fig:M-of-C-7}
\end{figure}
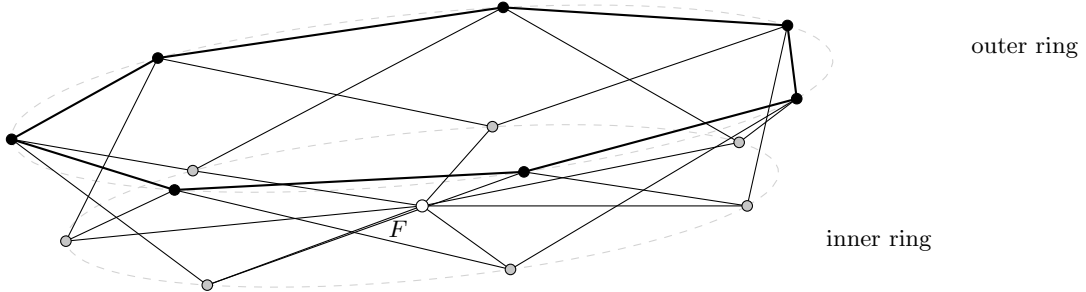

We now turn our attention to lower bounds.  It is straightforward to show that for $n\geq 3$, the dimension of $C_n$ is at least $2$, and hence $\dim(M(C_n))\geq 2$.  In Section \ref{sec:intro} we showed that $\dim(M(C_{10}))=2$.  By Lemma \ref{lem:upper-bound}, to complete the proof of Theorem \ref{thm:main}, we must show that for $n\neq 10$, we have that $\dim(M(C_n))\geq 3$.

\begin{proof}[Proof of Theorem \ref{thm:main}]
Suppose $n\geq 3$ is an integer with $n\neq 10$.  For the sake of contradiction, assume that $M(C_n)$ is isomorphic to a subgraph of $U_2$.  By a slight abuse of notation, we identify $M(C_n)$ with that subgraph.  That is, we regard $M(C_n)$ itself as a unit-distance graph.  As in examples \ref{Ex-C-5} and \ref{Ex-C-10}, we write $A_j$ for $(j,1)$ and $A'_j$ for $(j,0)$.

Observe that for all $i\in\mathbb{Z}/n\mathbb{Z}$, we have that $A_{i+1}A'_iFA'_{i+2}$ is a $4$-cycle in $M(C_n)$, as $A_{i+1}\sim A'_i$ and $A'_i\sim F$ and $F\sim A'_{i+2}$ and $A'_{i+2}\sim A_{i+1}$.  Now, $M(C_n)$ is a unit-distance graph, so a $4$-cycle in it is a rhombus of side length $1$.  Because opposite sides of a rhombus are parallel, two vector equations can be made: \begin{equation}\label{eq:1}
\overrightarrow{FA'_i} = \overrightarrow{A'_{i+2}A_{i+1}}\end{equation}\begin{equation}\label{eq:2}\overrightarrow{FA'_{i+2}} = \overrightarrow{A'_{i}A_{i+1}}\end{equation}We have not yet used the condition $n\neq 10$, and indeed, (\ref{eq:1}) and (\ref{eq:2}) can be seen in Figure \ref{fig:M-of-C-10}.  Because (\ref{eq:1}) holds for all $i$, we may replace $i$ with $i+2$ and then combine with (\ref{eq:2}) to get\begin{equation}\label{eq:3}
\overrightarrow{A'_{i+4}A_{i+3}} = \overrightarrow{A'_{i}A_{i+1}}\end{equation}for all $i\in\mathbb{Z}/n\mathbb{Z}$.

Similarly, the $4$-cycle $A_iA_{i+1}A_{i+2}A'_{i+1}$ leads to the equations\begin{equation}\label{eq:4}
\overrightarrow{A_iA_{i+1}} = \overrightarrow{A'_{i+1}A_{i+2}}
\end{equation}and\begin{equation}\label{eq:5}\overrightarrow{A_iA'_{i+1}} = \overrightarrow{A_{i+1}A_{i+2}}\end{equation}for all $i$.  Replacing $i$ with $i+1$ in (\ref{eq:4}) and combining with (\ref{eq:5}) gives us \begin{equation}\label{eq:6}\overrightarrow{A'_{i+2}A_{i+3}} = \overrightarrow{A_iA'_{i+1}}\end{equation}for all $i$.  Replacing $i$ with $i+2$ in (\ref{eq:3}) and substituting into (\ref{eq:6}) gives us \begin{equation}\label{eq:7}\overrightarrow{A'_{i+6}A_{i+5}} = \overrightarrow{A_{i}A'_{i+1}}=-\overrightarrow{A'_{i+1}A_i}\end{equation}for all $i$, where the second equality comes from reversing the direction of the vector.  Replacing $i$ with $i+5$ in (\ref{eq:7}) and then substituting into (\ref{eq:7}), we get \begin{equation}\label{eq:8}\overrightarrow{A'_{i+11}A_{i+10}} =-\overrightarrow{A'_{i+6}A_{i+5}}=\overrightarrow{A'_{i+1}A_i}\end{equation}for all $i$.  Substituting $i+9$ for $i$ in (\ref{eq:1}) and then combining with (\ref{eq:8}), we find that \begin{equation}\label{eq:9}
\overrightarrow{FA'_{i+9}} = \overrightarrow{A'_{i+11}A_{i+10}}=\overrightarrow{A'_{i+1}A_i}\end{equation}for all $i$.  Substituting $i-1$ for $i$ in (\ref{eq:1}) and then combining with (\ref{eq:9}), we find that \begin{equation}\label{eq:10}
\overrightarrow{FA'_{i+9}} =\overrightarrow{FA'_{i-1}}\end{equation}for all $i$.  Taking $i=1$ in (\ref{eq:10}) we get that $\overrightarrow{FA'_{10}} =\overrightarrow{FA'_{0}}$.  This is an equation of unit vectors with the same initial vertex, so it must be that $A'_{10}=A'_0$.  Thus $10\equiv 0\;(\text{mod }n)$, which implies that $n$ has to be a divisor of $10$.  We assumed $n\geq 3$ and $n\neq 10$, so $n=5$.

As a result, $A'_{i+6} = A'_{i+1}$ and $A_{i+5} = A_{i}$.  Equation (\ref{eq:7}) then gives us that $\overrightarrow{A'_{i+1}A_{i}} = -\overrightarrow{A'_{i+1}A_i}$ for all $i\in\mathbb{Z}/5\mathbb{Z}$.  But $\overrightarrow{A'_{i+1}A_{i}}$ is a unit vector, so this is a contradiction.\end{proof}

\section{A generalization}

The method of proof used in Theorem \ref{thm:main} can be generalized, as we now discuss.  Given a graph $G$, we define an auxiliary graph $\text{VecNeg}(G)$ as follows.  We regard $G$ as a directed graph, with a directed edge from $A$ to $B$ whenever $A$ and $B$ are adjacent in $G$.  We denote the edge from $A$ to $B$ by $\overrightarrow{AB}$.  The vertex set of $\text{VecNeg}(G)$  is the set of directed edges of $G$.  Recall that a \emph{walk} in a graph is a sequence $(v_0,v_1,\dots v_n)$ of vertices, where $v_{i-1}$ is adjacent to $v_i$ for all $i=1,\dots,n$.  We sometimes omit the parentheses and commas when writing a walk.  The \emph{length} of the walk $(v_0,v_1,\dots v_n)$ is $n$, and this walk is said to be \emph{closed} if $v_0=v_n$.  The vertices $\overrightarrow{AB}$ and $\overrightarrow{CD}$ of $\text{VecNeg}(G)$ are adjacent if and only if (i) $C=B$ and $D=A$, or (ii) the vertices $A$, $B$, $C$ and $D$ are distinct, and $ABCDA$ is a closed walk of length $4$ in $G$.  Recall that a finite graph is \emph{nonbipartite} if and only if it contains a closed walk of odd length.

\begin{Thm}\label{thm:generalization}
Let $G$ be a finite graph.  Then $\dim G\geq 3$ if either\begin{enumerate}

\item\label{item:first} The graph $\text{VecNeg}(G)$ is nonbipartite, or

\item\label{item:second} For some distinct vertices $F, A, B$ of $G$, where $A$ and $B$ are adjacent to $F$, there exists in $\text{VecNeg}(G)$ a walk of even length from $\overrightarrow{FA}$ to $\overrightarrow{FB}$.
\end{enumerate}
 
\end{Thm}\begin{proof}
Suppose that $G$ is isomorphic to a subgraph $G'$ of $U_2$.  By a slight abuse of notation, we identify $G$ with $G'$, and every directed edge in $G$ with the corresponding unit vector in $\mathbb{R}^2$.

We first observe that if $\overrightarrow{AB}$ and $\overrightarrow{CD}$ are adjacent in $\text{VecNeg}(G)$, then $\overrightarrow{AB}=-\overrightarrow{CD}$.  (Hence the notation ``VecNeg.'')  Here's why.  If $C=B$ and $D=A$, then the result is immediate.  Otherwise, we have that the vertices $A$, $B$, $C$ and $D$ are distinct, and $ABCDA$ is a closed walk of length $4$ in $G$.  This implies that $ABCD$ forms a rhombus of side length $1$, and the result follows.

If (\ref{item:first}) holds, then there is a closed walk $(e_0,\dots,e_n)$ of odd length in $\text{VecNeg}(G)$.  Here each $e_i$ is a directed edge in $G$.  Using the result of the previous paragraph, we get that $e_{i-1}=-e_i$ for all $i=1,\dots,n$.  Hence $e_0=(-1)^ne_n=(-1)^ne_0$.  This is a contradiction, because $n$ is odd, and $e_0$ is a unit vector.

If (\ref{item:second}) holds, then by similar reasoning we obtain $\overrightarrow{FA}=\overrightarrow{FB}$.  But then $A=B$, a contradiction.\end{proof}

The line of reasoning formalized in the proof of this theorem can be found in \cite{hong2014euclidean}, \cite{globus2019small}, and \cite{alexeev2025erdhosunitdistanceproblem}.

The reader who wishes to do so can recast the proof of Theorem \ref{thm:main} in terms of Theorem \ref{thm:generalization}.  The central idea is to convert chains of vector equations into walks in the graph $\text{VecNeg}(G)$.

We now give an example to illustrate a direct usage of Theorem \ref{thm:generalization}.

\begin{Ex}
For an integer $k\geq 3$, the M\"obius ladder graph $M_k$ is defined as follows.  It has $2k$ vertices, which we call $A_1,\dots,A_k,B_1,\dots,B_k$.  The vertex $A_j$ is adjacent to $B_j$ for all $j=1,\dots k$.  For all $j=1,\dots k-1$, we have that $A_j$ is adjacent to $A_{j+1}$ and $B_j$ is adjacent to $B_{j+1}$.  Finally, $A_1$ is adjacent to $B_k$, and $A_k$ is adjacent to $B_1$.  There are no other edges.  Figure \ref{fig:Mobius} shows the graph $M_5$, which may make apparent the reason for the name ``M\"obius ladder": it resembles a ladder that loops back around on itself but with a half-twist, à la a M\"obius strip.

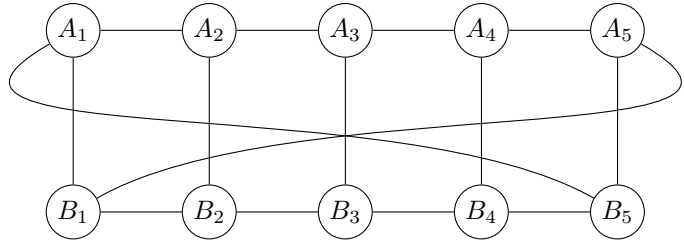
\begin{figure}[ht]
\centering
\begin{tikzpicture}[
    every node/.style={circle, draw, fill=white, inner sep=2pt},
    scale=1.2
]

\node (A1) at (0,2) {$A_1$};
\node (A2) at (1.5,2) {$A_2$};
\node (A3) at (3,2) {$A_3$};
\node (A4) at (4.5,2) {$A_4$};
\node (A5) at (6,2) {$A_5$};

\node (B1) at (0,0) {$B_1$};
\node (B2) at (1.5,0) {$B_2$};
\node (B3) at (3,0) {$B_3$};
\node (B4) at (4.5,0) {$B_4$};
\node (B5) at (6,0) {$B_5$};

\foreach \j in {1,...,5}
    \draw (A\j) -- (B\j);

\foreach \j/\k in {1/2,2/3,3/4,4/5}
{
    \draw (A\j) -- (A\k);
    \draw (B\j) -- (B\k);
}

\draw (A1) to[out=210,in=150] (B5);
\draw (A5) to[out=-30,in=30] (B1);

\end{tikzpicture}
\caption{The M\"obius ladder graph $M_5$}
\label{fig:Mobius}\end{figure}

In \cite{hong2014euclidean} it is shown that $\dim(M_3)\geq 3$.  (Indeed, that was how the authors learned about the main technique used in this paper.)  We now show that $\dim(M_k)\geq 3$ for all integers $k\geq 3$.  In $\text{VecNeg}(M_k)$, we have a closed walk of length $2k-1$, namely: \[\left(\overrightarrow{A_1B_1},\,\overrightarrow{B_2A_2},\,\overrightarrow{A_2B_2},\,\overrightarrow{B_3A_3},\dots,\overrightarrow{A_{k-1}B_{k-1}},\,\overrightarrow{B_kA_k},\,\overrightarrow{A_kB_k},\,\overrightarrow{A_1B_1}\right).\]To see that this is indeed a walk, note that for $j=1,\dots, k-1$, we have that $\overrightarrow{A_jB_j}$ and $\overrightarrow{B_{j+1}A_{j+1}}$ are adjacent in $\text{VecNeg}(M_k)$ because $(A_j,B_j,B_{j+1},A_{j+1},A_j)$ is a closed walk of length $4$ in $M_k$.  Similarly, $(A_k,B_k,A_1,B_1,A_k)$ is a closed walk of length $4$ in $M_k$, giving us the final step in the walk.  Because $2k-1$ is odd, it follows that $\text{VecNeg}(M_k)$ is nonbipartite.  By Theorem \ref{thm:generalization}, therefore $\dim(M_k)\geq 3$.\end{Ex}
 
\bibliographystyle{amsplain}
\bibliography{mycielski-bib}
 
\end{document}